\documentclass[a4paper,12pt]{article}
\usepackage{amsmath}
\usepackage{amsthm}
\usepackage{amssymb}
\usepackage{mathrsfs}
\usepackage{enumerate}
\usepackage{url}
\usepackage[utf8]{inputenc}

\title{A relationship between three-dimensional Bessel processes in terms of Pitman's transformation}

\author{Yuu Hariya\thanks{Supported in part by JSPS KAKENHI Grant Number  22K03330}}

\date{\empty}
\numberwithin{equation}{section}

\theoremstyle{plain}

\newtheorem{theorem}{Theorem}[section]
\newtheorem{lemma}[theorem]{Lemma}
\newtheorem{proposition}[theorem]{Proposition}

\theoremstyle{definition}

\theoremstyle{remark}
\newtheorem{remark}[theorem]{Remark}

\begin{document}

\newcommand\ND{\newcommand}
\newcommand\RD{\renewcommand}

\ND\R{\mathbb{R}}

\ND\ind{\boldsymbol{1}}

\ND\ga{\gamma}
\ND\la{\lambda }
\ND\ve{\varepsilon}

\ND\lref[1]{Lemma~\ref{#1}}
\ND\tref[1]{Theorem~\ref{#1}}
\ND\pref[1]{Proposition~\ref{#1}}
\ND\sref[1]{Section~\ref{#1}}
\ND\ssref[1]{Subsection~\ref{#1}}
\ND\aref[1]{Appendix~\ref{#1}}
\ND\rref[1]{Remark~\ref{#1}} 
\ND\cref[1]{Corollary~\ref{#1}}
\ND\csref[1]{Corollaries~\ref{#1}}
\ND\eref[1]{Example~\ref{#1}}
\ND\fref[1]{Fig.\ {#1} }
\ND\lsref[1]{Lemmas~\ref{#1}}
\ND\tsref[1]{Theorems~\ref{#1}}
\ND\dref[1]{Definition~\ref{#1}}
\ND\psref[1]{Propositions~\ref{#1}}
\ND\rsref[1]{Remarks~\ref{#1}}
\ND\sssref[1]{Subsections~\ref{#1}}

\ND\pr{\mathbb{P}}
\ND\ex{\mathbb{E}}

\ND\eqd{\stackrel{(d)}{=}}

\ND\cm{\mathcal{M}}

\ND\cmb{\overline{\cm}}

\ND\cpb{\overline{\cp}}

\ND\cp{\mathcal{P}}
\ND\cq{\mathcal{Q}}

\ND\sg{\sigma}
\ND\ctd[2]{C([0,#1];\mathbb{R}^{#2})}

\ND\pbes[2]{\mathbb{P}^{(3)}_{#1,#2}}
\ND\ebes[2]{\mathbb{E}^{(3)}_{#1,#2}}

\ND\bb{\beta}

\ND\hB{\widehat{\beta}}
\ND\tB{\widetilde{B}}
\ND\scp{\mathscr{P}}
\ND\scq{\mathscr{Q}}
\ND\bes{\mathrm{BES}^{3}}

\ND\bd{\mathbf{B}}

\ND\cl{\mathcal{L}}

\ND{\rmid}[1]{\mathrel{}\middle#1\mathrel{}}

\def\thefootnote{{}}

\maketitle 
\begin{abstract}
For every $x\ge 0$, let $R^{x}=\{ R^{x}_{t}\} _{t\ge 0}$ be a three-dimensional 
Bessel process started at $x$. Given $x>0$, with the aid of Pitman's theorem, 
a path transformation $T_{x}$ under which the process $T_{x}(R^{y})$ is distributed as 
$R^{x}$ for any $y>0$ is constructed. In fact, we prove that, 
for every $x,y>0$, the two three-dimensional processes 
\begin{align*}
 \bigl( R^{x},T_{y}(R^{x}),\mathcal{P}(-R^{x})+x\bigr) && \text{and} && 
 \bigl( T_{x}(R^{y}),R^{y},\mathcal{P}(-R^{y})+y\bigr) 
\end{align*}
are identical in law, where $\mathcal{P}$ is Pitman's transformation. In the course of 
our exploration, it will be revealed that, for any $x>0$,  
the process $\mathcal{P}(-R^{x})+x$ has the same law as the standard three-dimensional 
Bessel process $R^{0}$, which is then applied to derive a disintegration formula 
for the law of $R^{x}$ in terms of Pitman's transformation. An application of the 
formula to three-dimensional Bessel bridges is also presented.
\footnote{Mathematical Institute, Tohoku University, Aoba-ku, Sendai 980-8578, Japan}
\footnote{E-mail: hariya@tohoku.ac.jp}
\end{abstract}

\section{Introduction and main result}\label{;intro}

Along with Brownian motion, the three-dimensional Bessel process is a 
fundamental family of stochastic processes in probability theory. It is 
defined as the modulus of a three-dimensional Brownian motion and its 
pinned process, namely a three-dimensional Bessel bridge, started at 
and conditioned to end at the origin at a prescribed time, is a key ingredient 
in the excursion theory of one-dimensional Brownian motion; see, e.g., 
\cite[Sections~VI.3 and XII.2]{ry} and references therein. Let $B=\{ B_{t}\} _{t\ge 0}$ 
be a one-dimensional standard Brownian motion. Perhaps the most striking 
feature of the three-dimensional Bessel process would be the well-known fact 
\cite{jwp} that, when started at the origin, it is realized as the Pitman transform 
$\cp (B)$ of $B$: 
\begin{align}\label{;pitb}
 \cp (B)(t)=2\max _{0\le s\le t}B_{s}-B_{t},\quad t\ge 0; 
\end{align}
more precisely, one has  
\begin{align}\label{;pit}
 \Bigl\{ 
 \Bigl( 
 R_{t},-R_{t}+2\inf _{s\ge t}R_{s}
 \Bigr) 
 \Bigr\} _{t\ge 0}  
 \eqd 
 \bigl\{ 
 \bigl( \cp (B)(t),B_{t}\bigr) 
 \bigr\} _{t\ge 0}, 
\end{align}
where $R=\{ R_{t}\} _{t\ge 0}$ is a standard three-dimensional Bessel 
process and the equality stands for the identity in law. 
See Theorem~1.3 in \cite{jwp} and the description of the theorem 
that follows its statement; we also refer to 
\cite[Chapter~VI, Theorem~\thetag{3.5}]{ry}, as well as \rref{;rpit} below. 

For every $x\ge 0$, we let $R^{x}=\{ R^{x}_{t}\} _{t\ge 0}$ denote a three-dimensional 
Bessel process started at $x$. When $x=0$, we simply write $R$ for $R^{0}$ as 
above. In this paper, with the aid of Pitman's identity~\eqref{;pit} in law, 
we construct a family of path transformations that connects any pair 
of three-dimensional Bessel processes $R^{x}$ and $R^{y}$ with distinct 
starting points $x,y>0$.

Let $C(\R _{+};\R )$ be the space of real-valued continuous functions 
over the half-line $\R _{+}=[0,\infty )$ equipped with the Borel $\sg $-field 
relative to the topology of local uniform convergence. For every $z>0$ and 
$\phi \in C(\R _{+};\R )$ such that $\phi _{t}>0$ for all $t\ge 0$, we define 
\begin{equation}\label{;trz}
\begin{split}
 &T_{z}(\phi )(t)\\
 &:=\phi _{t}-2\left\{ 
 \left( 
 z-\phi _{0}-\frac{z}{\phi _{0}}\inf _{s\ge 0}\phi _{s}+\min _{0\le s\le t}\phi _{s}
 \right) \wedge 
 \Bigl( 
 \inf _{s\ge t}\phi _{s}-\min _{0\le s\le t}\phi _{s}
 \Bigr) _{+}
 \right\} +z-\phi _{0}
\end{split}
\end{equation}
for $t\ge 0$. Here and in what follows, for every $a,b\in \R $, we set 
\begin{align*}
 a\vee b:=\max \{ a,b\} , && a\wedge b:=\min \{ a,b\} ,
\end{align*}
and $a_{+}:=a\vee 0$. Notice that, if $z=\phi _{0}$, then 
$T_{z}(\phi )(t)=\phi _{t}$ for all $t\ge 0$ since we have  
\begin{align*}
 &\Bigl( 
 -\inf _{s\ge 0}\phi _{s}+\min _{0\le s\le t}\phi _{s}
 \Bigr) 
 \wedge 
 \Bigl( 
 \inf _{s\ge t}\phi _{s}-\min _{0\le s\le t}\phi _{s}
 \Bigr) _{+}\\
 &=\Bigl\{ 
 -\Bigl( \inf _{s\ge t}\phi _{s}\wedge \min _{0\le s\le t}\phi _{s}\Bigr) 
 +\min _{0\le s\le t}\phi _{s}
 \Bigr\} 
 \wedge 
 \Bigl( 
 \inf _{s\ge t}\phi _{s}-\min _{0\le s\le t}\phi _{s}
 \Bigr) _{+}\\
 &=
 \Bigl( 
 -\inf _{s\ge t}\phi _{s}+\min _{0\le s\le t}\phi _{s}
 \Bigr) _{+}
 \wedge 
 \Bigl( 
 \inf _{s\ge t}\phi _{s}-\min _{0\le s\le t}\phi _{s}
 \Bigr) _{+},
\end{align*}
which is $0$ for any $t\ge 0$. With Pitman's transformation 
$\cp \colon C(\R _{+};\R )\to C(\R _{+};\R )$ as exhibited in \eqref{;pitb}, 
we also define, following the notation in \cite{har25+} by the author, 
\begin{equation}\label{;cq}
\begin{split}
 \cq (\phi )(t)&:=\cp (-\phi )(t)+\phi _{0}\\
 &\equiv \phi _{t}-2\min _{0\le s\le t}\phi _{s}+\phi _{0},
\end{split}
\end{equation} 
for $\phi \in C(\R _{+};\R )$ and $t\ge 0$.
In the sequel, the symbol $U$ denotes a random variable 
uniformly distributed on the interval $[0,1]$. 
One of the main results of the paper is stated as follows:

\begin{theorem}\label{;tmain}
Given $x,y>0$, the following two random variables  
valued in the product space $\R \times C(\R _{+};\R )^{3}$ have the same 
law: 
\begin{align*}
 \left( 
 1-\frac{1}{x}\inf _{s\ge 0}R^{x}_{s},R^{x},T_{y}(R^{x}),\cq (R^{x})
 \right) , 
 && 
 \left( 
 1-\frac{1}{y}\inf _{s\ge 0}R^{y}_{s},T_{x}(R^{y}),R^{y},\cq (R^{y})
 \right) . 
\end{align*}
In fact, they are both identical in law with 
\begin{align*}
 \left( 
 U,
 \Bigl\{ 
 R_{t}-2\Bigl( 
 (xU)\wedge \inf _{s\ge t}R_{s}
 \Bigr) +x
 \Bigr\} _{t\ge 0},
 \Bigl\{ 
 R_{t}-2\Bigl( 
 (yU)\wedge \inf _{s\ge t}R_{s}
 \Bigr) +y
 \Bigr\} _{t\ge 0},
 R
 \right) ,
\end{align*}
where $U$ is assumed to be independent of $R$.
In particular, for any $x>0$, it holds that 
\begin{align*}
 \cq (R^{x})\eqd R
\end{align*}
and that $\cq (R^{x})$ is independent of $\inf _{s\ge 0}R^{x}_{s}$. Moreover, 
we have 
\begin{align}\label{;infs}
 \inf _{s\ge 0}T_{y}(R^{x})(s)=\frac{y}{x}\inf _{s\ge 0}R^{x}_{s}, \quad 
 \text{a.s.,}
\end{align}
for any $x,y>0$.
\end{theorem}

To facilitate the understanding of the theorem, we discuss below the 
special case with $y=2x$ and compare the result with a path decomposition 
of the three-dimensional Bessel process due to D.~Williams. 

Given $x>0$, let $\phi ^{x}\in C(\R _{+};\R )$ be such that 
$\phi ^{x}_{t}>0$ for all $t\ge 0$ and $\phi ^{x}_{0}=x$. 
Then it is readily seen that 
\begin{align}\label{;tr2x}
 T_{2x}(\phi ^{x})(t)
 =\phi ^{x}_{t}-2\Bigl( 
 x\wedge \inf _{s\ge t}\phi ^{x}_{s}-\min _{0\le s\le t}\phi ^{x}_{s}
 \Bigr) _{+}+x
\end{align}
for all $t\ge 0$. Indeed, for every fixed $t\ge 0$, by the 
definition~\eqref{;trz} of transformations $T_{z}$, 
\begin{align*}
 T_{2x}(\phi ^{x})(t)
 =\phi ^{x}_{t}-2\bigl\{ 
 \bigl( 
 x-2(a\wedge b)+b
 \bigr) \wedge (a-b)_{+}
 \bigr\} +x,
\end{align*}
where we set 
\begin{align*}
 a=\inf _{s\ge t}\phi ^{x}_{s}, && 
 b=\min _{0\le s\le t}\phi ^{x}_{s},
\end{align*}
so that the right-hand side of \eqref{;tr2x} is expressed as 
\begin{align*}
 \phi ^{x}_{t}-2(x\wedge a-b)_{+}+x.
\end{align*}
Noting that $b\le x$, we divide the study into the  
two cases $a\le b$ and $b<a$. In the case $a\le b$, we have 
\begin{align*}
 \bigl( 
 x-2(a\wedge b)+b
 \bigr) \wedge (a-b)_{+}
 &=(x-2a+b)\wedge 0\\
 &=0
\end{align*}
since $x+b\ge 2b\ge 2a$, while 
\begin{align*}
 (x\wedge a-b)_{+}&=(a-b)_{+}\\
 &=0.
\end{align*}
In the case $b<a$, 
\begin{align*}
 \bigl( 
 x-2(a\wedge b)+b
 \bigr) \wedge (a-b)_{+}
 &=(x-b)\wedge (a-b)\\
 &=x\wedge a-b\\
 &=(x\wedge a-b)_{+}
\end{align*}
since $x\wedge a\ge b$. Therefore the expression~\eqref{;tr2x} of 
$T_{2x}(\phi ^{x})$ is verified. Consequently, by \tref{;tmain}, 
the process described by 
\begin{align}\label{;t2xr}
 \begin{cases}
 R^{x}_{t}+x, & 0\le t\le \sg ^{x},\\
 \displaystyle R^{x}_{t}-2\inf _{s\ge t}R^{x}_{s}+2\inf _{s\ge 0}R^{x}_{s}+x, & 
 \sg ^{x}\le t\le L^{x},\\
 \displaystyle R^{x}_{t}+2\inf _{s\ge 0}R^{x}_{s}-x, & t\ge L^{x},
 \end{cases}
\end{align}
has the same law as $R^{2x}$, where we set 
\begin{align*}
 \sg ^{x}:=\inf \Bigl\{ 
 t\ge 0;\,R^{x}_{t}=\inf _{s\ge 0}R^{x}_{s}
 \Bigr\} , && 
 L^{x}:=\sup \!\left\{ 
 t\ge 0;\,R^{x}_{t}=x
 \right\} ,
\end{align*}
namely, $\sg ^{x}$ is the first time at which $R^{x}$ attains its global infimum 
and $L^{x}$ is the last exit time of $R^{x}$ from level $x$. 
Notice that both $\sg ^{x}$ and $L^{x}$ are finite a.s.\ since 
$\lim _{t\to \infty }R^{x}_{t}=\infty $ a.s. We may compare the above observation 
with Williams' path decomposition of the three-dimensional Bessel process 
which we recall below. Prepare four independent elements 
$U,\beta ,\hB ,\rho $, where 
$\beta =\{ \beta _{t}\} _{t\ge 0}$ and $\hB =\{ \hB _{t}\} _{t\ge 0}$ are two 
one-dimensional standard Brownian motions, and 
$\rho =\{ \rho _{t}\} _{t\ge 0}$ is a standard three-dimensional Bessel process. 
For a one-dimensional continuous process $X=\{ X_{t}\} _{t\ge 0}$ and $a\in \R $, 
we write $\tau _{a}(X)$ for the first hitting time of $X$ to $a$:
\begin{align*}
 \tau _{a}(X):=\inf \{ t\ge 0;\,X_{t}=a\} ,
\end{align*}
with the convention that $\inf \emptyset =\infty $.
Then, for every $x>0$, the process given by 
\begin{align}\label{;wd}
 \begin{cases}
 x-\beta _{t}, & 0\le t\le \tau _{xU}(\beta ),\\
 x-\hB _{\tau _{xU}(\beta )+\tau _{xU}(\hB )-t}, & 
 \tau _{xU}(\beta )\le t\le \tau _{xU}(\beta )+\tau _{xU}(\hB ),\\
 x+\rho _{t-\tau _{xU}(\beta )-\tau _{xU}(\hB )}, & t\ge \tau _{xU}(\beta )+\tau _{xU}(\hB ),
 \end{cases}
\end{align}
is a three-dimensional Bessel process started at $x$; see, e.g., \cite[Section~VI.3]{ry}. 
If we apply the 
above decomposition~\eqref{;wd} with $x$ replaced by $2x$, then, as is observed 
in \eqref{;t2xr}, the fact that $T_{2x}(R^{x})$, which is distributed as $R^{2x}$, only 
differs from $R^{x}$ by $x$ up to time $\sg ^{x}$ is consistent with the decomposition. 
Moreover, we know from \eqref{;wd} that the time at which a three-dimensional 
Bessel process attains its global infimum is a.s.\ unique, and that 
\begin{align}\label{;infx}
 \sg ^{x}\eqd \tau _{xU}(\beta ),
\end{align}
which, together with the fact that 
\begin{align*}
 T_{2x}(R^{x})(L^{x})&=2\inf _{s\ge 0}R^{x}_{s}\\
 &=\inf _{s\ge 0}T_{2x}(R^{x})(s)
\end{align*}
a.s.\ by \eqref{;t2xr} and \eqref{;infs}, implies 
\begin{align}\label{;infl}
 \sg ^{2x}\eqd L^{x}.
\end{align}
The last identity in law is consistent with \eqref{;infx} in such a way that 
\begin{align*}
 \sg ^{2x}&\eqd \tau _{2xU}(\beta )\\
 &\eqd \tau _{xU}(\beta )+\tau _{xU}(\hB ),
\end{align*}
which, in view of the path decomposition~\eqref{;wd}, 
is identical in law with $L^{x}$, 
where the second line is due to the strong Markov property of Brownian 
motion and the independence of $\beta ,\hB $ and $U$.

We turn to the accompanying transform $T_{x}(R^{2x})$ of $R^{2x}$ in our theorem. 
Similar reasoning to the above reveals that 
\begin{align*}
 T_{x}(R^{2x})(t)
 =R^{2x}_{t}+\Bigl( 
 2x+\inf _{s\ge 0}R^{2x}_{s}-2\min _{0\le s\le t}R^{2x}_{s}
 \Bigr) _{+}-x
\end{align*}
for all $t\ge 0$. Therefore, by setting 
\begin{align*}
 \lambda ^{2x}
 =\min \Bigl\{ 
 t\ge 0;\,R^{2x}_{t}=x+\frac{1}{2}\inf _{s\ge 0}R^{2x}_{s}
 \Bigr\} ,
\end{align*}
we see that the process described by 
\begin{align*}
 \begin{cases}
 R^{2x}_{t}-x, & 0\le t\le \lambda ^{2x},\\
 \displaystyle R^{2x}_{t}-2\min _{0\le s\le t}R^{2x}_{s}+\inf _{s\ge 0}R^{2x}_{s}+x, & \lambda ^{2x}\le t\le \sg ^{2x},\\
 \displaystyle R^{2x}_{t}-\inf _{s\ge 0}R^{2x}_{s}+x, & t\ge \sg ^{2x},
 \end{cases}
\end{align*}
has the same law as $R^{x}$. Since 
\begin{align*}
 T_{x}(R^{2x})(\lambda ^{2x})
 &=\frac{1}{2}\inf _{s\ge 0}R^{2x}_{s}\\
 &=\inf _{s\ge 0}T_{x}(R^{2x})(s)
\end{align*}
a.s.\ by \eqref{;infs}, and since 
\begin{align*}
 R^{2x}_{t}-\inf _{s\ge 0}R^{2x}_{s}+x
 &\ge R^{2x}_{\sg ^{2x}}-\inf _{s\ge 0}R^{2x}_{s}+x\\
 &=x
\end{align*}
for all $t\ge \sg ^{2x}$, the aforementioned a.s.\ uniqueness of the time at which 
a three-dimensional Bessel process attains its global infimum entails 
\begin{align*}
 \left( 
 \lambda ^{2x},\sg ^{2x}
 \right) \eqd 
 \left( 
 \sg ^{x},L^{x}
 \right) ,
\end{align*}
extending \eqref{;infl}. 

As a summary of the above discussion, we emphasize that 
our \tref{;tmain} shows that, for every $x>0$, the joint law of the 
three processes 
\begin{align*}
 R^{x}_{t},\ R^{x}_{t}-2\Bigl( 
 x\wedge \inf _{s\ge t}R^{x}_{s}-\min _{0\le s\le t}R^{x}_{s}
 \Bigr) _{+}+x,\ \cq (R^{x})(t), \quad t\ge 0,
\end{align*}
is the same as that of 
\begin{align*}
 R^{2x}_{t}+\Bigl( 
 2x+\inf _{s\ge 0}R^{2x}_{s}-2\min _{0\le s\le t}R^{2x}_{s}
 \Bigr) _{+}-x,\ 
 R^{2x}_{t},\ \cq (R^{2x})(t), \quad t\ge 0.
\end{align*}
Finally, in a general case where $y>0$ is arbitrary, 
we observe that 
\begin{align}\label{;infy}
 \sg ^{y}\eqd 
 \begin{cases}
 \displaystyle 
 \min \Bigl\{ 
 t\ge 0;\,R^{x}_{t}=x-y+\frac{y}{x}\inf _{s\ge 0}R^{x}_{s}
 \Bigr\} , & y\le x,\\[3mm]
 \displaystyle 
 \sup \Bigl\{ 
 t\ge 0;\,R^{x}_{t}=y-x-\left( 
 \frac{y}{x}-2
 \right) \inf _{s\ge 0}R^{x}_{s}
 \Bigr\} , & y\ge x,
 \end{cases}
\end{align}
which is \eqref{;infl} when $y=2x$.
To verify the above identity~\eqref{;infy} in law, we let $\sg _{*}$ denote its 
right-hand side. Then it holds that 
\begin{align*}
 \sg _{*}\le \sg ^{x}\ \text{for $y\le x$}
 && \text{and} && 
 \sg _{*}\ge \sg ^{x}\ \text{for $y\ge x$}.
\end{align*}
Indeed, 
\begin{align*}
 \inf _{s\ge 0}R^{x}_{s}\le x-y+\frac{y}{x}\inf _{s\ge 0}R^{x}_{s}\le x&\quad \text{for $y\le x$},\\
\intertext{and}
 y-x-\left( 
 \frac{y}{x}-2
 \right) \inf _{s\ge 0}R^{x}_{s}\ge \inf _{s\ge 0}R^{x}_{s}&\quad \text{for $y\ge x$},
\end{align*}
because of the fact that 
\begin{align*}
 1-\frac{1}{x}\inf _{s\ge 0}R^{x}_{s}\ge 0.
\end{align*}
Therefore, by the definition~\eqref{;trz} of the transformations $T_{z}$, in the case $y\le x$, 
\begin{align*}
 T_{y}(R^{x})(\sg _{*})
 &=R^{x}_{\sg _{*}}-2\Bigl\{ 
 \Bigl( 
 y-x-\frac{y}{x}\inf _{s\ge 0}R^{x}_{s}+R^{x}_{\sg _{*}}
 \Bigr) \wedge 0
 \Bigr\} +y-x\\
 &=x-y+\frac{y}{x}\inf _{s\ge 0}R^{x}_{s}+y-x\\
 &=\frac{y}{x}\inf _{s\ge 0}R^{x}_{s},
\end{align*}
while in the case $y\ge x$, we also see that 
\begin{align*}
 T_{y}(R^{x})(\sg _{*})
 &=R^{x}_{\sg _{*}}-2\Bigl\{ 
 \Bigl( 
 y-x-\frac{y}{x}\inf _{s\ge 0}R^{x}_{s}+\inf _{s\ge 0}R^{x}_{s}
 \Bigr) \wedge 
 \Bigl( 
 R^{x}_{\sg _{*}}-\inf _{s\ge 0}R^{x}_{s}
 \Bigr) 
 \Bigr\} +y-x\\
 &=y-x-\left( 
 \frac{y}{x}-2
 \right) \inf _{s\ge 0}R^{x}_{s}
 -2\Bigl( 
 y-x-\frac{y}{x}\inf _{s\ge 0}R^{x}_{s}+\inf _{s\ge 0}R^{x}_{s}
 \Bigr) +y-x\\
 &=\frac{y}{x}\inf _{s\ge 0}R^{x}_{s},
\end{align*}
two of which verify \eqref{;infy} in view of \eqref{;infs}.

We give an outline of the paper. In \sref{;sptmain}, we prove 
\tref{;tmain}. A key to the proof is \pref{;prxx}, which is 
also applied in \sref{;scl} to obtain a disintegration 
formula for the law of $R^{x}$ for $x>0$ 
in terms of Pitman's transformation 
along with an application to three-dimensional 
Bessel bridges.

\section{Proof of \tref{;tmain}}\label{;sptmain}

Pick $s>0$ arbitrarily and let $x>0$. The proof of \tref{;tmain} hinges upon a 
simple observation on the conditional law 
of the process $\{ \cp (B)(t+s)\} _{t\ge 0}$ given $\cp (B)(s)=x$ 
which we will describe below. To that end, 
by setting 
\begin{align*}
 \tB _{t}:=B_{t+s}-B_{s},\quad t\ge 0,
\end{align*}
we rewrite, for each $t\ge 0$, 
\begin{align*}
 \cp (B)(t+s)&=2\Bigl( 
 \max _{0\le u\le s}B_{u}\vee \max _{s\le u\le t+s}B_{u}
 \Bigr) -B_{t+s}\\
 &=2\Bigl\{ 
 \max _{0\le u\le s}B_{u}\vee \Bigl( 
 \max _{0\le u\le t}\tB _{u}+B_{s}
 \Bigr) 
 \Bigr\} -\tB _{t}-B_{s}\\
 &=\Bigl( 
 \cp (B)(s)-\tB _{t}
 \Bigr) \vee 
 \Bigl( 
 \cp (\tB )(t)+B_{s}
 \Bigr) .
\end{align*}
Notice that, conditionally on $\cp (B)(s)=x$, $B_{s}$ is uniformly distributed 
on $[-x,x]$, that is, identical in law with $x(2U-1)$. 
Indeed, thanks to Pitman's identity~\eqref{;pit}, we have in particular 
\begin{align*}
 \bigl( 
 \cp (B)(s),B_{s}
 \bigr) 
 \eqd 
 \Bigl( 
 R_{s},-R_{s}+2\inf _{u\ge s}R_{u}
 \Bigr) ;
\end{align*}
then the above claim immediately follows by the well-known fact that, 
as may also be deduced from \eqref{;wd}, 
\begin{align}\label{;unif}
 \inf _{u\ge 0}R^{x}_{u}\eqd xU, 
\end{align}
and by the Markov property of the three-dimensional Bessel process.
Since $2U-1\eqd 1-2U$ and the process $\tB =\{ \tB _{t}\} _{t\ge 0}$ 
is a Brownian motion independent of $(\cp (B)(s),B_{s})$, we arrive at 

\begin{lemma}\label{;lrxb}
For every $x>0$, it holds that 
\begin{align}\label{;rxb}
 R^{x}
 \eqd 
 \bigl\{ 
 (x-B_{t})\vee \bigl( \cp (B)(t)+x(1-2U)\bigr) 
 \bigr\} _{t\ge 0},
\end{align}
where $U$ is assumed to be independent of $B$.
\end{lemma}

\begin{proof}
By \eqref{;pit} and the Markov property of the three-dimensional Bessel process, 
the conditional law of the process $\{ \cp (B)(t+s)\} _{t\ge 0}$ given $\cp (B)(s)=x$ 
coincides with the law of $R^{x}$.
\end{proof}

For every fixed $x>0$, we write $X^{x}=\{ X^{x}_{t}\} _{t\ge 0}$ for the 
process in the right-hand side of \eqref{;rxb}. Then it is readily seen that 
\begin{align}\label{;xx}
 X^{x}_{t}
 =
 \begin{cases}
 x-B_{t}, & 0\le t\le \tau _{xU}(B),\\
 \cp (B)(t)+x(1-2U), & t\ge \tau _{xU}(B), 
 \end{cases}
\end{align}
since, for every $t\ge 0$,
\begin{align*}
 \cp (B)(t)+x(1-2U)+B_{t}=2\max _{0\le s\le t}B_{s}+x(1-2U),
\end{align*}
which exceeds $x$ at $t=\tau _{xU}(B)$. By the fact that 
$\cp (B)(t)\ge \max _{0\le s\le t}B_{s}$ for all $t\ge 0$, 
we have 
\begin{align*}
 \inf _{s\ge 0}X^{x}_{s}&=x(1-U)\\
 &=X^{x}_{\tau _{xU}(B)},
\end{align*}
and hence inspection of the path decomposition~\eqref{;wd} of $R^{x}$ tells us that 
\begin{align}\label{;joint}
 \Bigl( \sg ^{x},\inf _{s\ge 0}R^{x}_{s},\{ R^{x}_{t}\} _{t\ge 0}\Bigr) 
 \eqd 
 \bigl( 
 \tau _{xU}(B),x(1-U),\{ X^{x}_{t}\} _{t\ge 0}
 \bigr) .
\end{align}

\begin{lemma}\label{;lrxx}
For every $x>0$, it holds that 
\begin{align}\label{;rxx}
 \bigl\{ 
 \bigl( 
 R^{x}_{t},\cq (R^{x})(t)
 \bigr) 
 \bigr\} _{t\ge 0}
 \eqd 
 \bigl\{ 
 \bigl( 
 X^{x}_{t},\cp (B)(t)
 \bigr) 
 \bigr\} _{t\ge 0}.
\end{align}
\end{lemma}

\begin{proof}
By the description~\eqref{;xx} of the process $X^{x}$, we have 
\begin{align*}
 \cp (B)(t)
 =
 \begin{cases}
 \cp (x-X^{x})(t), & 0\le t\le \tau _{xU}(B),\\
 X^{x}_{t}-2x(1-U)+x, & t\ge \tau _{xU}(B).
 \end{cases}
\end{align*}
Therefore, by \eqref{;joint}, we see that the two-dimensional process 
consisting of 
\begin{align*}
 R^{x} && \text{and} && 
 \begin{cases}
 \cp (x-R^{x})(t), & 0\le t\le \sg ^{x},\\
 \displaystyle R^{x}_{t}-2\inf _{s\ge 0}R^{x}_{s}+x, & t\ge \sg ^{x},
 \end{cases}
\end{align*}
has the same law as the two-dimensional process 
\begin{align*}
 \bigl( 
 X^{x}_{t},\cp (B)(t)
 \bigr) ,\quad t\ge 0.
\end{align*}
Since $\sg ^{x}$ is the time that $R^{x}$ attains $\inf _{s\ge 0}R^{x}_{s}$, 
the conclusion follows by the definition~\eqref{;cq} of the transformation $\cq $.
\end{proof}

Combining the above lemma with Pitman's identity~\eqref{;pit} in law, 
we obtain a key proposition towards the proof of \tref{;tmain} as 
follows:

\begin{proposition}\label{;prxx}
For every $x>0$, the following two two-dimensional processes 
\begin{align*}
 \bigl( 
 R^{x}_{t},\cq (R^{x})(t)
 \bigr) ,\quad t\ge 0,
\end{align*}
and 
\begin{align*}
 \Bigl( 
 R_{t}-2\Bigl( 
 (xU)\wedge \inf _{s\ge t}R_{s}
 \Bigr) +x,R_{t}
 \Bigr) ,\quad t\ge 0,
\end{align*}
are identical in law. Here, $U$ is independent of $R$.
\end{proposition}

\begin{proof}
Recall that $X^{x}$ denotes the right-hand side of \eqref{;rxb}. 
Then, by \eqref{;pit} and the independence of $U$ and $R$, 
the right-hand side of \eqref{;rxx} is identical in law with the 
two-dimensional process 
\begin{align*}
 \Bigl( 
 \Bigl( 
 R_{t}+x-2\inf _{s\ge t}R_{s}
 \Bigr) \vee (R_{t}+x-2xU),
 R_{t}
 \Bigr) ,\quad t\ge 0,
\end{align*}
from which the claim follows readily.
\end{proof}

By the relation that $2(a\wedge b)=a+b-|a-b|$ for $a,b\in \R $, 
notice that   
\begin{align*}
 &R_{t}-2\Bigl( 
 (xU)\wedge \inf _{s\ge t}R_{s}
 \Bigr) +x\\
 &=R_{t}-xU -\inf _{s\ge t}R_{s}+\Bigl| 
 xU-\inf _{s\ge t}R_{s}
 \Bigr| +x\\
 &\ge x(1-U)
\end{align*}
for all $t\ge 0$ and the equality holds at the last exit time of $R$ 
from $xU$.

In order to prove \tref{;tmain}, we also use the following 
deterministic lemma: 

\begin{lemma}\label{;ldet}
For every $\phi \in C(\R _{+};\R )$, one has 
\begin{align*}
 \inf _{s\ge t}\cp (\phi )(s)
 =\max _{0\le s\le t}\phi _{s}+\Bigl( 
 \max _{0\le s\le t}\phi _{s}-\sup _{s\ge t}\phi _{s}
 \Bigr) _{+}
\end{align*}
for all $t\ge 0$.
\end{lemma}

\begin{proof}
Fix $\phi \in C(\R _{+};\R )$ arbitrarily. We know from \cite[Lemma~2.2]{har25+} 
that, for every $T>0$, 
\begin{align*}
 \min _{t\le s\le T}\cp (\phi )(s)
 =\max _{0\le s\le t}\phi _{s}+\Bigl( 
 \max _{0\le s\le t}\phi _{s}-\max _{t\le s\le T}\phi _{s}
 \Bigr) _{+}
\end{align*}
for all $0\le t\le T$. Letting $T\to \infty $ on both sides proves the claim.
\end{proof}

\begin{remark}\label{;rpit}
The above lemma shows that, if $\phi \in C(\R _{+};\R )$ is such that 
$\lim _{t\to \infty }\phi _{t}=\infty $, then 
\begin{align*}
 -\cp (\phi )(t)+2\inf _{s\ge t}\cp (\phi )(s)=\phi _{t}
\end{align*}
for all $t\ge 0$. Therefore, once we know the identity in law between 
$R$ and $\cp (B)$, then the joint identity~\eqref{;pit} follows by the 
fact that $\lim _{t\to \infty }R_{t}=\infty $ a.s.
\end{remark}

We are in a position to prove \tref{;tmain}. 

\begin{proof}[Proof of \tref{;tmain}]
Fix $x,y>0$ arbitrarily below. For every $z,z'\ge 0$ and for every 
$\phi \in C(\R _{+};\R )$ such that $\phi _{t}\ge 0$ for all $t\ge 0$, 
we set 
\begin{align*}
 \Xi ^{(z)}_{t}(z',\phi )
 :=\phi _{t}-2\Bigl( 
 (zz')\wedge \inf _{s\ge t}\phi _{s}
 \Bigr) +z,\quad t\ge 0,
\end{align*}
so that, for every $z>0$, 
\begin{align}\label{;prod}
 \bigl( 
 I^{z},R^{z},\cq (R^{z})
 \bigr) \eqd 
 \bigl( 
 U,\Xi ^{(z)}(U,R),R
 \bigr) 
\end{align}
by \pref{;prxx}, where we put 
$I^{z}=1-(1/z)\inf _{s\ge 0}R^{z}_{s}$. We regard each side of \eqref{;prod} as 
a random variable valued in the product space $\R \times C(\R _{+};\R )^{2}$;  
a similar remark applies below. Then we have 
\begin{align*}
 &\bigl( 
 U,\Xi ^{(x)}(U,R),\Xi ^{(y)}(U,R),R
 \bigr) \\
 &\eqd \bigl( 
 I^{x},R^{x},\Xi ^{(y)}(I^{x},\cq (R^{x})),\cq (R^{x})
 \bigr) \\
 &\eqd \bigl( 
 I^{y},\Xi ^{(x)}(I^{y},\cq (R^{y})),R^{y},\cq (R^{y})
 \bigr) ,
\end{align*}
where we have used \eqref{;prod} for the second and third lines 
with $z=x$ and $z=y$, respectively. As $x$ and $y$ are arbitrarily, 
it now remains to prove that 
\begin{align}\label{;transs}
 \Xi ^{(y)}_{t}(I^{x},\cq (R^{x}))
 =T_{y}(R^{x})(t)
\end{align}
for every $t\ge 0$. By the definition~\eqref{;cq} 
of the transformation $\cq $, the left-hand side of \eqref{;transs} is 
equal to 
\begin{align*}
 &R^{x}_{t}-2\min _{0\le s\le t}R^{x}_{s}+x
 -2\Bigl( 
 (yI^{x})\wedge \inf _{s\ge t}\bigl( 
 \cp (-R^{x})(s)+x
 \bigr) 
 \Bigr) +y\\
 &=R^{x}_{t}-2\Bigl\{  
 \Bigl( 
 yI^{x}-x+\min _{0\le s\le t}R^{x}_{s}
 \Bigr) \wedge \Bigl( 
 \inf _{s\ge t}\cp (-R^{x})(s)+\min _{0\le s\le t}R^{x}_{s}
 \Bigr) 
 \Bigr\} +y-x,
\end{align*}
which agrees with $T_{y}(R^{x})(t)$ 
thanks to \lref{;ldet} and hence completes the proof of the theorem.
\end{proof}

\section{A few consequences of \pref{;prxx}}\label{;scl}

Given $x>0$, \pref{;prxx} also enables us to investigate the law 
of $R^{x}$ conditioned on $\cp (-R^{x})$, or equivalently on $\cq (R^{x})$, 
which we will discuss in this section.

Unless otherwise specified, we fix $x>0$ throughout. 
The independence of $\cp (-R^{x})$ and $\inf _{s\ge 0}R^{x}_{s}$ as 
observed in \tref{;tmain} reflects the fact that, 
in general, Pitman's transformation reduces the information of the original 
path. A simple example is provided by the pair $\phi ^{1}_{t}=t$ and 
$\phi ^{2}_{t}=-t$ for $t\ge 0$, for which we have 
$\cp (\phi ^{1})(t)=\cp (\phi ^{2})(t)=t$ for all $t\ge 0$. This lack of 
information of $\cp $ leads to an interesting consequence when 
applied to a specific stochastic process. For instance, we have observed at 
the beginning of \sref{;sptmain} that, for every $t>0$, conditionally 
on $\cp (B)(t)=x$, $B_{t}$ is uniformly distributed on $[-x,x]$, which, 
thanks to Pitman's identity~\eqref{;pit} and the Markov property of 
the three-dimensional Bessel process, extends to the fact that, conditionally on 
$\scp _{s},\,s\le t$, with $\cp (B)(t)=x$, the law of $B_{t}$ is 
still the uniform distribution over $[-x,x]$. Here, $\{ \scp _{t}\} _{t\ge 0}$ 
is the natural filtration of the process $\{ \cp (B)(t)\} _{t\ge 0}$: 
\begin{align*}
 \scp _{t}:=\sg \!\left( 
 \cp (B)(s),s\le t
 \right) .
\end{align*}
In fact, as discussed in \cite{har25+}, we can say more as will be seen below. 

We fix $t>0$ in the sequel. We let $C([0,t];\R )$ 
be the space of real-valued continuous functions over the interval $[0,t]$. 
We equip $C([0,t];\R )$ with the topology of uniform convergence. 
For every $z\in \R $, we define 
$\cm _{z}\colon C([0,t];\R )\to C([0,t];\R )$ by 
\begin{equation*}
\begin{split}
 &\cm _{z}(\phi )(s)\\
 &:=\phi _{s}-\frac{\phi _{t}-z}{2}-\left| 
 \frac{\phi _{t}-z}{2}+\max _{0\le u\le s}\phi _{u}-\max _{s\le u\le t}\phi _{u}
 \right| 
 +\Bigl| 
 \max _{0\le u\le s}\phi _{u}-\max _{s\le u\le t}\phi _{u}
 \Bigr| 
\end{split}
\end{equation*}
for $0\le s\le t$ and $\phi \in C([0,t];\R )$, which, by \cite[Proposition~2.1]{har25+}, 
admits the following expression in terms of $\cp $ for 
every $\phi \in C([0,t];\R )$ and $0\le s\le t$:
\begin{align}\label{;itpit}
 \cm _{z}(\phi )(s)
 =-\cp (\phi )(s)+\min \Bigl\{ 
 2\min _{s\le u\le t}\cp (\phi )(u),\cp (\phi )(t)+z
 \Bigr\} ,
\end{align}
where we regard $\cp $ as a transformation on $\ctd{t}{}$ 
keeping the same notation. Note that, by definition, 
\begin{align}\label{;itd}
 \cm _{\phi _{t}}(\phi )=\phi 
\end{align}
for all $\phi \in \ctd{t}{}$. 
In the next statement, as well as in what follows, 
a real-valued function defined on a topological space is 
said to be measurable if it is Borel measurable.

\begin{proposition}[{\cite[Proposition~2.3]{har25+}}]\label{;ppit}
For every nonnegative measurable function $F$ on $\ctd{t}{}$, 
it holds that 
\begin{align*}
 \ex [F(B)]=\int _{\R }dz\,\ex \!\left[ 
 \frac{1}{2\cp (B)(t)}F\bigl( 
 \cm _{z}(B)
 \bigr) ;\,\cp (B)(t)\ge |z|
 \right] .
\end{align*}
\end{proposition}

In fact, since the functional $F\bigl( \cm _{z}(B)\bigr) $ of $B$ is 
$\scp _{t}$-measurable in view of the expression~\eqref{;itpit} of 
$\cm _{z}$ for every $z\in \R $, the above proposition shows 
that 
\begin{align}\label{;cpit}
 \ex \!\left[ 
 F(B) \rmid| \scp _{t}
 \right] 
 =\frac{1}{2\cp (B)(t)}\int _{\{ |z|\le \cp (B)(t)\} }dz\,F\bigl( 
 \cm _{z}(B)
 \bigr) ,\quad \text{a.s.};
\end{align}
notice that the conditional expectation on the left-hand side is 
well-defined thanks to the conditional monotone convergence theorem 
even if $F(B)$ is not integrable. Also indicated is that, on the event that 
$\cp (B)(t)\ge |z|$, we have $\cm _{z}(B)(t)=z$ a.s., which is indeed 
the case as the following lemma shows:

\begin{lemma}[{\cite[Lemmas~2.5 and 2.7]{har25+}}]\label{;ltv}
For every pair $(z,\phi )\in \R \times \ctd{t}{}$ satisfying 
\begin{align*}
 |z-\phi _{0}|\le \cp (\phi )(t)-\phi _{0}, 
\end{align*}
we have 
\begin{enumerate}[(i)]{}
 \item[\thetag{i}] $\cm _{z}(\phi )(0)=\phi _{0}$ and $\cm _{z}(\phi )(t)=z$;
 
 \item[\thetag{ii}] $\cp (\cm _{z}(\phi ))=\cp (\phi )$ and in particular, by expression~\eqref{;itpit},
 \begin{align*}
  \cm _{z'}(\cm _{z}(\phi ))=\cm _{z'}(\phi )
 \end{align*}
 for all $z'\in \R $.
\end{enumerate}
\end{lemma}

Consequently, from \eqref{;cpit}, we see in particular that, for every nonnegative 
measurable function $f$ on $\R $, 
\begin{align*}
 \ex \!\left[ 
 f(B_{t})\rmid| \scp _{t}
 \right] 
 =\frac{1}{2\cp (B)(t)}\int _{\{ |z|\le \cp (B)(t)\} }dz\,f(z)
 ,\quad \text{a.s.}
\end{align*}
In this section, we apply \pref{;prxx} to develop similar results for 
$\{ R^{x}_{s}\} _{0\le s\le t}$ given $\cq (R^{x})$ up to time $t$.

Define 
\begin{align*}
 \scq ^{x}_{s}:=\sg \bigl( 
 \cq (R^{x})(u),u\le s
 \bigr) ,\quad s\ge 0.
\end{align*}
Pick $y>0$ arbitrarily. First we discuss the conditional law of the 
terminal value $R^{x}_{t}$ given $\scq ^{x}_{t}$ with $\cq (R^{x})(t)=y$. 
To this end, we observe from \pref{;prxx} that the conditional in law 
in question is the same as the law of the random variable $\xi $ given by 
\begin{align*}
 \xi =x+y-2\bigl( 
 (xU)\wedge (yU')
 \bigr) ,
\end{align*}
where $U'$ is independent of $U$ and $U'\eqd U$, which is due to the 
fact that 
\begin{align*}
 \inf _{s\ge t}R_{s}\eqd yU
\end{align*}
given $\sg (R_{s},s\le t)$ with $R_{t}=y$, thanks to the Markov 
property of the three-dimensional Bessel process and identity~\eqref{;unif}. 
Notice that, a.s., 
\begin{align*}
 x+y&\ge \xi \\
 &\ge x+y-2(x\wedge y)\\
 &=|x-y|,
\end{align*}
and that, for every $|x-y|\le z\le x+y$, 
\begin{align*}
 \pr \!\left( 
 \xi \le z
 \right) &=\pr \Bigl(  
 (xU)\wedge (xU')\ge \frac{x+y-z}{2}
 \Bigr) \\
 &=\Bigl( 
 1-\frac{x+y-z}{2x}
 \Bigr) \!\Bigl( 
 1-\frac{x+y-z}{2y}
 \Bigr) \\
 &=\frac{z^{2}-(x-y)^{2}}{4xy}
\end{align*}
by the independence of $U$ and $U'$. Consequently, we obtain 

\begin{proposition}\label{;pcdd}
For every $y>0$, the conditional law of $R^{x}_{t}$ given 
$\scq ^{x}_{t}$ with $\cq (R^{x})(t)=y$ has density  
\begin{align*}
 \frac{z}{2xy}\ind _{[|x-y|,x+y]}(z),\quad z\in \R ;
\end{align*}
that is, for every nonnegative function $f$ on $\R $, 
\begin{align}\label{;cdd}
 \ex \!\left[ 
 f(R^{x}_{t})\rmid| \scq ^{x}_{t}
 \right] 
 =\frac{1}{2x\cq (R^{x})(t)}
 \int _{\left\{ 
 \left| 
 x-\cq (R^{x})(t)
 \right| \le z\le x+\cq (R^{x})(t)
 \right\} }dz\,zf(z),\quad \text{a.s.}
\end{align}
\end{proposition}

For every $z\in \R $, define, as in \cite{har25+}, 
$\cmb _{z}\colon \ctd{t}{}\to \ctd{t}{}$ by 
\begin{align}\label{;cmbz}
 \cmb _{z}(\phi ):=-\cm _{-z}(-\phi ),\quad 
 \phi \in \ctd{t}{},
\end{align}
so that 
\begin{equation*}
 \cmb _{z}(\phi )(s)=\phi _{s}-\frac{\phi _{t}-z}{2}+\left| 
 \frac{\phi _{t}-z}{2}+\min _{0\le u\le s}\phi _{u}-\min _{s\le u\le t}\phi _{u}
 \right| 
 -\Bigl| 
 \min _{0\le u\le s}\phi _{u}-\min _{s\le u\le t}\phi _{u}
 \Bigr| 
\end{equation*}
for all $0\le s\le t$ and $\phi \in \ctd{t}{}$. Noting that 
the relation 
\begin{align*}
 \left| 
 x-\cq (R^{x})(t)
 \right| \le z\le x+\cq (R^{x})(t)
\end{align*}
in \eqref{;cdd} is equivalently stated as 
\begin{align}\label{;event}
 |x-z|\le \cq (R^{x})(t)\le x+z,
\end{align}
we obtain the following disintegration formula for the law of 
$\{ R^{x}_{s}\} _{0\le s\le t}$ similar to \pref{;ppit}, 
which is of interest in its own right:

\begin{proposition}\label{;pdisint}
For every nonnegative measurable function $F$ on 
$\ctd{t}{}$, we have 
\begin{align}\label{;disint}
 \ex \!\left[ 
 F(R^{x})
 \right]
 =\int _{\R _{+}}dz\,\frac{z}{2x}
 \ex \!\left[ 
 \frac{1}{\cq (R^{x})(t)}
 F\bigl( 
 \cmb _{z}(R^{x})
 \bigr) ;\,|x-z|\le \cq (R^{x})(t)\le x+z
 \right] , 
\end{align}
or, equivalently, 
\begin{align*}
 \ex \!\left[ 
 F(R^{x})\rmid| \scq ^{x}_{t}
 \right] 
 =\frac{1}{2x\cq (R^{x})(t)}
 \int _{\left\{ 
 \left| 
 x-\cq (R^{x})(t)
 \right| \le z\le x+\cq (R^{x})(t)
 \right\} }dz\,zF\bigl( 
 \cmb _{z}(R^{x})
 \bigr) ,\quad \text{a.s.}
\end{align*}
\end{proposition}

Notice that, on the event~\eqref{;event}, we have, a.s.,  
\begin{align}\label{;tv}
 \cmb _{z}(R^{x})(0)=x && \text{and} && \cmb _{z}(R^{x})(t)=z.
\end{align}
In fact, regarding $\cq$ as a transformation on 
$\ctd{t}{}$ as well, we have the following restatement of \lref{;ltv} in terms of 
$\cmb _{z}$ and $\cq $: 

\begin{lemma}\label{;ltvd}
For every pair $(z,\phi )\in \R \times \ctd{t}{}$ satisfying 
\begin{align*}
 |z-\phi _{0}|\le \cq (\phi )(t), 
\end{align*}
we have 
\begin{enumerate}[(i)]{}
 \item[\thetag{i}] $\cmb _{z}(\phi )(0)=\phi _{0}$ and $\cmb _{z}(\phi )(t)=z$;
 
 \item[\thetag{ii}] $\cq \bigl( \cmb _{z}(\phi )\bigr) =\cq (\phi )$ and 
 $\cmb _{z'}\bigl( \cmb _{z}(\phi )\bigr) =\cmb _{z'}(\phi )$ for all $z'\in \R $.
\end{enumerate}
\end{lemma}
As a consequence, \pref{;pcdd} may be recovered from 
\pref{;pdisint} by taking $F(\phi )=f(\phi _{t}),\,\phi \in \ctd{t}{}$.

\pref{;pdisint} has an application to a conditional identity in law 
for three-dimensional Bessel bridges with a 
common starting point. For every $y,z\ge 0$, we designate by 
$\pbes{y}{z}$ the law of a three-dimensional Bessel bridge from $y$ to $z$ over 
$[0,t]$, under which we write $r=\{ r_{s}\} _{0\le s\le t}$ for the coordinate process 
in $\ctd{t}{}$: 
\begin{align*}
 r_{s}(\phi ):=\phi _{s},\quad 0\le s\le t,\ \phi \in \ctd{t}{}. 
\end{align*}
For the definition of Bessel bridges of positive dimension, 
see \cite[Section~XI.3]{ry}. We also write $\ebes{y}{z}$ for the expectation 
with respect to $\pbes{y}{z}$. Pick a Borel-measurable subset $\Gamma $ 
of $\R _{+}$ arbitrarily and consider in formula~\eqref{;disint} the 
function $F$ of the form 
\begin{align*}
 F(\phi )=\ind _{\Gamma }(\phi _{t})G(\phi ),\quad \phi \in \ctd{t}{},
\end{align*} 
with $G$ a bounded and continuous function. By the second 
equality in \eqref{;tv}, it then follows that 
\begin{equation}\label{;disintd}
\begin{split}
 &\int _{\Gamma }\ebes{x}{z}\!\left[ 
 G(r)
 \right] \pr (R^{x}_{t}\in dz)\\
 &=\int _{\Gamma }dz\,\frac{z}{2x}
 \ex \!\left[ 
 \frac{1}{\cq (R^{x})(t)}
 G\bigl( 
 \cmb _{z}(R^{x})
 \bigr) ;\,|x-z|\le \cq (R^{x})(t)\le x+z
 \right] .
\end{split}
\end{equation}
Notice that, thanks to the boundedness and continuity of 
$G$, the function 
\begin{align*}
 \R _{+}\ni z\mapsto 
 \ebes{x}{z}\!\left[ 
 G(r)
 \right] 
\end{align*}
is continuous since the mapping 
$\R _{+}\ni z\mapsto \pbes{x}{z}$ is continuous with respect to 
the topology of weak convergence; see, e.g., \cite[Subsection~2.2]{cub}. 
In view of the definition~\eqref{;cmbz} of transformations $\cmb _{z}$ and \pref{;prxx}, 
it is not difficult to see that the function 
\begin{align*}
 \R _{+}\ni z\mapsto  \ex \!\left[ 
 \frac{1}{\cq (R^{x})(t)}
 G\bigl( 
 \cmb _{z}(R^{x})
 \bigr) ;\,|x-z|\le \cq (R^{x})(t)\le x+z
 \right] 
\end{align*}
is also continuous thanks to the boundedness and continuity 
of $G$. Recall the fact that 
\begin{align*}
 \pr (R^{x}_{t}\in dz)
 =\sqrt{\frac{2}{\pi t}}\frac{z}{x}\exp \left( 
 -\frac{x^{2}+z^{2}}{2t}
 \right) \sinh \frac{xz}{t}\,dz,\quad z\ge 0
\end{align*}
(see, e.g., \cite[Sections~VI.3 and XI.1]{ry} and \cite[p.~619]{cub}). Summing up the above 
observations, we see from \eqref{;disintd} that 
\begin{align*}
 &\ebes{x}{z}\!\left[ 
 G(r)
 \right] \sqrt{\frac{2}{\pi t}}\frac{z}{x}\exp \left( 
 -\frac{x^{2}+z^{2}}{2t}
 \right) \sinh \frac{xz}{t}\\
 &=\frac{z}{2x}
 \ex \!\left[ 
 \frac{1}{\cq (R^{x})(t)}
 G\bigl( 
 \cmb _{z}(R^{x})
 \bigr) ;\,|x-z|\le \cq (R^{x})(t)\le x+z
 \right] ,
\end{align*}
which holds true for all $z\ge 0$, not only for a.e.\ $z\ge 0$, by the 
continuity of both sides in $z$. By the standard density argument, 
the last relation extends to 

\begin{proposition}\label{;pbridge}
For every $z\ge 0$, it holds that 
\begin{align*}
 &\ebes{x}{z}\!\left[ 
 F(r)
 \right] \sqrt{\frac{2}{\pi t}}\left\{ 
 \exp \left( 
 -\frac{|x-z|^{2}}{2t}
 \right) 
 -\exp \left( 
 -\frac{(x+z)^{2}}{2t}
 \right) 
 \right\} 
 \\
 &=\ex \!\left[ 
 \frac{1}{\cq (R^{x})(t)}
 F\bigl( 
 \cmb _{z}(R^{x})
 \bigr) ;\,|x-z|\le \cq (R^{x})(t)\le x+z
 \right] 
\end{align*}
for any nonnegative measurable function $F$ on $\ctd{t}{}$.
\end{proposition}

Taking $F\equiv 1$, we see in particular that 
\begin{align*}
 &\ex \!\left[ 
 \frac{1}{\cq (R^{x})(t)};\,|x-z|\le \cq (R^{x})(t)\le x+z
 \right] \\
 &=\sqrt{\frac{2}{\pi t}}\left\{ 
 \exp \left( 
 -\frac{|x-z|^{2}}{2t}
 \right) 
 -\exp \left( 
 -\frac{(x+z)^{2}}{2t}
 \right) 
 \right\} 
\end{align*}
for all $z\ge 0$, which is indeed the case since, by the 
identity $\cq (R^{x})(t)\eqd R_{t}$ from \pref{;prxx}, the left-hand side 
is equal to 
\begin{align}\label{;sbes}
 \int _{\left\{ 
 |x-z|\le y\le x+z
 \right\} }\frac{1}{y}\pr \!\left( 
 R_{t}\in dy
 \right) 
 =\sqrt{\frac{2}{\pi t^{3}}}\int _{|x-z|}^{x+z}dy\,y\exp \left( 
 -\frac{y^{2}}{2t}
 \right) 
\end{align}
(see, e.g., \cite[p.~252]{ry} and \cite[p.~619]{cub}), 
which agrees with the right-hand side. \pref{;pbridge} reveals a 
connection among three-dimensional Bessel bridges 
with distinct end points as in the following theorem, which may 
be regarded as a complement to \cite[Theorem~3.1]{har25+}.

\begin{theorem}\label{;tbridge}
Suppose that the two positive real numbers $z$ and $z'$ 
fulfill 
\begin{align}\label{;assump}
 z'\le z && \text{and} && |x-z|\le |x-z'|.
\end{align}
Then the three-dimensional process 
\begin{align*}
 \bigl( 
 r,\cq (r),\cmb _{z}(r)
 \bigr) 
\end{align*}
under $\pbes{x}{z'}$ has the same law as the conditioned process 
\begin{align*}
 \bigl( 
 \cmb _{z'}(r),\cq (r),r
 \bigr) 
\end{align*}
on the event that 
\begin{align}\label{;eventd}
 \frac{z-z'}{2}\le \min _{0\le s\le t}r_{s}\le \frac{x+z-|x-z'|}{2}
\end{align}
under $\pbes{x}{z}$.
\end{theorem}

The above theorem is not covered by \cite[Theorem~3.1]{har25+}
whose statement reduces to a trivial case if we consider a common 
starting point.

Notice that, in \eqref{;eventd}, we have 
\begin{align*}
 \frac{z-z'}{2}<\frac{x+z-|x-z'|}{2},
\end{align*}
since 
\begin{align*}
 x+z-|x-z'|-(z-z')&=x+z'-|x-z'|\\
 &>0
\end{align*}
by the positivity of $x$ and $z'$. As an example of \tref{;tbridge}, 
we let $z,z'>0$ be such that $z'\le z$ and $x=(z+z')/2$ so that 
\eqref{;assump} is fulfilled. Then, \eqref{;eventd} is reduced to 
\begin{align*}
 \frac{z-z'}{2}\le \min _{0\le s\le t}r_{s}\le \frac{z+z'}{2},
\end{align*}
for which the second inequality is satisfied $\pbes{(z+z')/2}{z}$-a.s. 
Therefore we conclude from \tref{;tbridge} that, for any 
nonnegative measurable function $F$ on $\ctd{t}{}^{3}$, 
\begin{align*}
 &\ebes{(z+z')/2}{z'}\!\left[ 
 F\bigl( 
 r,\cq (r),\cmb _{z}(r)
 \bigr) 
 \right] \\
 &=\ebes{(z+z')/2}{z}\!\left[ 
 F\bigl( 
 \cmb _{z'}(r),\cq (r),r
 \bigr) \rmid| \min _{0\le s\le t}r_{s}\ge \frac{z-z'}{2}
 \right] .
\end{align*}

\begin{proof}[Proof of \tref{;tbridge}]
For every $z\ge 0$, we write, with $x$ and $t$ being fixed, 
\begin{align*}
 q(z)=\sqrt{\frac{2}{\pi t}}\left\{ 
 \exp \left( 
 -\frac{|x-z|^{2}}{2t}
 \right) 
 -\exp \left( 
 -\frac{(x+z)^{2}}{2t}
 \right) 
 \right\} 
\end{align*}
for simplicity. Pick $z,z'>0$ fulfilling condition~\eqref{;assump}. 
Let $F$ be any nonnegative measurable function on $\ctd{t}{}$. 
By \pref{;pbridge}, we have 
\begin{align*}
 \ebes{x}{z'}\!\left[ 
 F(r)
 \right] q(z')
 =\ex \!\left[ 
 \frac{1}{\cq (R^{x})(t)}
 F\bigl( 
 \cmb _{z'}(R^{x})
 \bigr) ;\,|x-z'|\le \cq (R^{x})(t)\le x+z'
 \right] .
\end{align*}
Noting the inclusion 
\begin{align*}
\left\{ |x-z'|\le \cq (R^{x})(t)\le x+z'\right\} 
\subset 
\left\{ |x-z|\le \cq (R^{x})(t)\le x+z\right\} 
\end{align*}
and that, a.s.\ on the right-hand event, 
\begin{align*}
 \cmb _{z'}\bigl( 
 \cmb _{z}(R^{x})
 \bigr) 
 =\cmb _{z'}(R^{x}) && \text{and} && 
 \cq \bigl( 
 \cmb _{z}(R^{x})
 \bigr) 
 =\cq (R^{x})
\end{align*}
by \lref{;ltvd}, we have  
\begin{align}\label{;rel}
 \ebes{x}{z'}\!\left[ 
 F(r)
 \right] q(z')
 =\ebes{x}{z}\!\left[ 
 F\bigl( 
 \cmb _{z'}(r)
 \bigr) ;\,|x-z'|\le \cq (r)(t)\le x+z'
 \right] q(z)
\end{align}
by \pref{;pbridge} applied to the functional 
\begin{align*}
 F\bigl( \cmb _{z'}(\phi )\bigr) 
 \ind _{\left[ |x-z'|,x+z'\right] }\bigl( \cq (\phi )(t)\bigr) ,
 \quad \phi \in \ctd{t}{}.
\end{align*}
We write $A_{z,z'}$ for the event~\eqref{;eventd} so that 
\begin{align*}
 \left\{ 
 |x-z'|\le \cq (r)(t)\le x+z'
 \right\} =A_{z,z'},\quad \text{$\pbes{x}{z}$-a.s.},
\end{align*}
by the definition~\eqref{;cq} of the transformation $\cq $. Then, 
for $F$ is arbitrary, relation~\eqref{;rel} entails that the 
three-dimensional process 
\begin{align*}
 \bigl( 
 r,\cq (r),\cmb_{z}(r)
 \bigr) 
\end{align*}
under $\pbes{x}{z'}$ is identical in law with 
\begin{align*}
 \bigl( 
 \cmb _{z'}(r),
 \cq \bigl( 
 \cmb _{z'}(r)
 \bigr) ,
 \cmb _{z}\bigl( 
 \cmb _{z'}(r)
 \bigr) 
 \bigr) 
\end{align*}
under the conditional law $\pbes{x}{z}(\,\cdot \mid A_{z,z'})$. 
Since, a.s.\ on the event $A_{z,z'}$, 
\begin{align*}
 \cq \bigl( 
 \cmb _{z'}(r)
 \bigr) =\cq (r), && 
 \cmb _{z}\bigl( 
 \cmb _{z'}(r)
 \bigr) =\cmb _{z}(r),
\end{align*}
by \lref{;ltvd}, we have the conclusion of the theorem 
thanks to the fact that $\cmb _{z}(r)=r$ $\pbes{x}{z}$-a.s.\ in view of 
\eqref{;itd} and the definition~\eqref{;cmbz} of the transformation $\cmb _{z}$.
\end{proof}

We conclude this paper with a remark on relation~\eqref{;rel} in the above proof, 
as well as on another consequence of \pref{;prxx} of interest.

\begin{remark}
\thetag{1} Relation~\eqref{;rel} tells us that, for every $z,z'>0$ fulfilling 
\eqref{;assump}, 
\begin{align*}
 \pbes{x}{z}\!\left( 
 \frac{z-z'}{2}\le \min _{0\le s\le t}r_{s}\le \frac{x+z-|x-z'|}{2}
 \right) 
 &=\frac{q(z')}{q(z)}\\
 &=\exp \left( 
 \frac{z^{2}-(z')^{2}}{2t}
 \right) \frac{\sinh \frac{xz'}{t}}{\sinh \frac{xz}{t}}.
\end{align*}
Notice that $q(z')/q(z)\le 1$ since, for every $z\ge 0$, 
$q(z)$ is nothing but the right-hand side of \eqref{;sbes}.

\thetag{2} It would also be noteworthy that, by \pref{;prxx} and 
\cite[Corollary~1.2]{har25+}, the process given by 
\begin{align*}
 R^{x}_{t-s}+R^{x}_{t}-2\min _{t-s\le u\le t}R^{x}_{u},\quad 0\le s\le t,
\end{align*}
is also a standard three-dimensional Bessel process. Here, 
$x\ge 0$ and $t>0$ are arbitrarily taken.
\end{remark}


\end{document}